\documentclass[preprint,11pt]{elsarticle}

\usepackage{amsfonts, amsmath, amscd}
\usepackage[psamsfonts]{amssymb}

\usepackage{amssymb}

\usepackage{pb-diagram}

\usepackage[all,cmtip]{xy}

\usepackage{amsthm}

\usepackage[usenames]{color}

\newtheorem{theorem}{Theorem}[section]
\newtheorem{lemma}[theorem]{Lemma}
\newtheorem{corollary}[theorem]{Corollary}

\newtheorem{proposition}[theorem]{Proposition}

\newtheorem{problem}[theorem]{Problem}

\theoremstyle{definition}
\newtheorem{example}[theorem]{Example}
\newtheorem{definition}[theorem]{Definition}
\theoremstyle{remark}

\numberwithin{equation}{section}
\numberwithin{theorem}{section}

\newcommand{\e}{\varepsilon}
\newcommand{\w}{\omega}

\newcommand{\NN}{\mathbb{N}}

\newcommand{\RR}{\mathbb{R}}

\newcommand{\IR}{\mathbb{R}}

\newcommand{\AAA}{\mathcal{A}}

\newcommand{\Ra}{\Rightarrow}

\newcommand{\Int}{\mathsf{Int}}

\newcommand{\CC}{C_k}

\newcommand{\SM}{{\setminus}}

\def\om{\omega}

\def\Int{\operatorname{Int}}

\begin{document}

\begin{frontmatter}

\title{Baire-type properties of topological vector spaces}

\author{Saak Gabriyelyan}
\ead{saak@math.bgu.ac.il}
\address{Department of Mathematics, Ben-Gurion University of the Negev, Beer-Sheva, Israel}

\author{Alexander V. Osipov}
\ead{OAB@list.ru}
\address{Krasovskii Institute of Mathematics and Mechanics, Ural Federal University, Yekaterinburg, Russia}

\author{Evgenii Reznichenko}
\ead{erezn@inbox.ru}
\address{Department of Mathematics, Lomonosov Mosow State University, Moscow, Russia}

\begin{abstract}
Burzyk,  Kli\'{s} and Lipecki proved that every topological vector space (tvs) $E$ with the property $(K)$ is a Baire space.
K\c{a}kol and S\'{a}nchez Ruiz proved that every sequentially complete Fr\'{e}chet--Urysohn locally convex space (lcs) is  Baire.
Being motivated by the property $(K)$ and the notion of a Mackey null sequence we introduce a property $(MK)$ which is strictly weaker than the property $(K)$, and show that any locally complete lcs has the property $(MK)$.  We prove that any $\kappa$-Fr\'{e}chet--Urysohn tvs with the property $(MK)$ is a Baire space; consequently, each locally complete $\kappa$-Fr\'{e}chet--Urysohn lcs is a Baire space. This generalizes both the aforementioned results. We construct a feral Baire space $E$ with the property $(K)$ and which is not $\kappa$-Fr\'{e}chet--Urysohn.
Although a $\kappa$-Fr\'{e}chet--Urysohn lcs $E$ can be not a Baire space, we show that $E$ is always $b$-Baire-like in the sense of Ruess. Applications to spaces of Baire functions and $\CC$-spaces are given.

\end{abstract}


\end{frontmatter}


\section{Introduction}


One of the most important properties of topological vector spaces is the Baire property. Recall that a topological vector space (tvs for short) $E$ is {\em Baire} if the intersection of any sequence of open dense subsets of $E$ is dense in $E$. The starting point in the study of Baire topological vector spaces is the following classical result.

\begin{theorem} \label{t:cm-Baire}
Each complete and metrizable tvs $E$ is a Baire space.
\end{theorem}

The following characterization of a Baire tvs is due to Saxon \cite[Theorem~1]{Saxon}.

\begin{theorem} \label{t:Saxon}
A tvs $E$ is a Baire space if, and only if, every absorbing balanced and closed subset of $E$ is a neighborhood of some point.
\end{theorem}

Since the Baire property was studied in some classes of topological vector spaces or locally convex spaces which include complete and metrizable spaces.

The condition of completeness in Theorem \ref{t:cm-Baire} can be weaken to the property $(K)$ which is defined by the notion of a $K$-sequence. $K$-sequences in topological vector spaces had already been isolated by Mazur and Orlicz \cite{MazOr} and Alexiewicz \cite{Alex-50}.

\begin{definition}\label{def:ks}
A null sequence $(x_n)_n$ in a tvs $E$ is called a {\em $K$-sequence}  if every subsequence of $(x_n)_n$ contains a subsequence $(y_n)_n$ such that the series $\sum_n y_n$ converges in $E$.
\end{definition}

Then the property $(K)$ is defined as follows.

\begin{definition} \label{def:property-(K)}
A tvs $E$ has the {\em property $(K)$} if every null sequence in $E$ is a $K$-sequence.
\end{definition}

$K$-sequences and the property $(K)$ in the wider class of abelian metrizable groups were considered by  Antosik and Schwartz \cite{AntosikSchwartz1985} and by Burzyk,  Kli\'{s} and Lipecki \cite{BurzykLipecki1984}, see also Chapter 1 of the classical book  \cite{PB}. The following generalization of Theorem \ref{t:cm-Baire} follows from \cite[Theorem~2]{BurzykLipecki1984}.

\begin{theorem}[\cite{BurzykLipecki1984}] \label{t:m-K-Baire}
Each metrizable tvs $E$ with the property $(K)$ is a Baire space.
\end{theorem}

In the class of all locally convex spaces (lcs for short) Theorem \ref{t:cm-Baire} was generalized in two directions. The first one is to replace {\em completeness} by the strictly weaker condition of {\em sequential completeness}. Recall that a tvs $E$ is {\em sequentially complete} if each Cauchy sequence in $E$ converges to a point of $E$. The second direction is to weaken metrizability by the property of being a Fr\'{e}chet--Urysohn space. A Tychonoff space $X$ is {\em Fr\'{e}chet--Urysohn} if, for any $A\subseteq X$ and each $x\in \overline{A}$, there exists a sequence $\{a_n: n\in \w\}\subseteq A$ that converges to $x$. Topological vector spaces  which are Fr\'{e}chet--Urysohn were characterized by K\c{a}kol, L\'{o}pez-Pellicer and Todd \cite[Theorem~1]{KPLT}. The following generalization of Theorem \ref{t:cm-Baire} was proved by K\c{a}kol and S\'{a}nchez Ruiz \cite{KSR}.

\begin{theorem}[\cite{KSR}] \label{t:sc+FU=>Baire}
Every sequentially complete Fr\'{e}chet--Urysohn lcs is a Baire space.
\end{theorem}

The main purpose of this article is to obtain new sufficient conditions on a tvs $E$ under which $E$ is a Baire space and, at the same time, these conditions are strictly weaker than the property $(K)$, sequential completeness and Fr\'{e}chet--Urysohness considered in Theorems \ref{t:m-K-Baire} and \ref{t:sc+FU=>Baire}.

%
%
%

Let us start from a generalization of the property $(K)$. Let $E$ be a locally convex space. One of the most important classes of null sequences in $E$ is the class of Mackey null sequences. A sequence $\{x_n\}_{n\in\w}$ in $E$  is  {\em Mackey null}  if there is a disc $B\subseteq E$ such that $p_B(x_n)\to 0$, where $p_B$ is the Minkowski functional of $B$ (see \cite[2.10.17~Definition]{BS2017tvs}). It is well-known (see \cite[Lemma~2.10.18]{BS2017tvs}) that $\{x_n\}_{n\in\w}$ is Mackey null if, and only if, there exists an increasing unbounded sequence  $\{a_n\}_{n\in\w}$ of positive numbers such that $a_n x_n \to 0$ in $E$. This characterization can be extended to any topological vector space.

\begin{definition}\label{def:mc}
A  sequence $\{x_n\}_{n\in\w}$ in a tvs $E$ is said to be  {\em Mackey null}  if there exists an increasing unbounded sequence  $\{a_n\}_{n\in\w}$ of positive numbers such that $a_n x_n \to 0$ in $E$.
\end{definition}

Replacing ``null sequences'' by ``Mackey null sequences'' we obtain the following notion which is crucial in the proofs of the main results of the article.

\begin{definition}\label{def:MK-property}
A tvs $E$ has the {\em property $(MK)$} if every Mackey null sequence in $E$ is a $K$-sequence.
\end{definition}

Topological vector spaces with the property $(MK)$ are studied in Section \ref{sec:K-MK}. In Proposition \ref{p:K=>MK} we show that if a tvs $E$ has the property $(K)$, then it has the property $(MK)$. If, in addition, $E$ is Fr\'{e}chet--Urysohn, then $E$ has the property $(K)$ if, and only if, it has the property $(MK)$, see Proposition \ref{p:al4-MK=K}. The condition of being Fr\'{e}chet--Urysohn in Proposition \ref{p:al4-MK=K} is essential, see Example \ref{exa:MK-not=K} which states that the {\em complete} lcs $\IR^{\ell_2}$ has the property $(MK)$ but it does not have the property $(K)$. On the other hand, in Proposition \ref{p:lc-k-complete} we prove that each  locally complete lcs $E$ has the property $(MK)$. Recall that an lcs $L$ is {\em locally complete}  if the closed absolutely convex hull of any null sequence in $L$ is compact. For numerous equivalent conditions for an  lcs $L$ to be locally complete see Proposition 5.1.6 and Theorem 5.1.11 in \cite{PB}. It is known that any sequentially complete lcs is locally complete, but the converse assertion is not true in general. In Proposition \ref{p:feral-K} we show that any feral tvs $E$ (that is, any bounded subset of $E$ is finite-dimensional) has the property $(K)$.

Another direction to strengthen Theorem \ref{t:m-K-Baire} and Theorem \ref{t:sc+FU=>Baire} is to consider a topological property which is strictly weaker than the property of being a Fr\'{e}chet--Urysohn space. A considerably wider class of topological spaces than the class of Fr\'{e}chet--Urysohn spaces is the class of $\kappa$-Fr\'{e}chet--Urysohn spaces introduced by Arhangel'skii.

\begin{definition}\label{def:kFU}
A Tychonoff space $X$ is called {\em $\kappa$-Fr\'{e}chet--Urysohn} if, for any open set $U\subseteq X$ and each point $x\in \overline{U}$, there exists a sequence $\{x_n\}_{n\in \w}\subseteq U$ that converges to~$x$.
\end{definition}
\noindent Clearly, every Fr\'{e}chet--Urysohn space is $\kappa$-Fr\'{e}chet--Urysohn. It should be mentioned that without using of the term ``$\kappa$-Fr\'{e}chet--Urysohn'', Mr\'{o}wka proved in \cite{Mrowka} that any product of first countable spaces is  $\kappa$-Fr\'{e}chet--Urysohn. In particular, for every uncountable cardinal $\lambda$, the space $\IR^\lambda$ is $\kappa$-Fr\'{e}chet--Urysohn but not Fr\'{e}chet--Urysohn.

The main results of the article are proved in Section \ref{sec:kFU-Baire1}. In Theorem \ref{t:kFU-group-charac} we characterize $\kappa$-Fr\'{e}chet--Urysohn topological vector spaces. Using this characterization we prove the following theorem which is our principal result. 

\begin{theorem} \label{t:MK-Baire}
A $\kappa$-Fr\'{e}chet--Urysohn tvs $E$ with the property $(MK)$ is a Baire space.
\end{theorem}

The following corollary immediately follows from Theorem \ref{t:MK-Baire} and the aforementioned  Proposition \ref{p:lc-k-complete}.

\begin{corollary} \label{c:lc-kFU-Baire}
A locally complete $\kappa$-Fr\'{e}chet--Urysohn locally convex space is a Baire space.
\end{corollary}
\noindent It is clear that Theorem \ref{t:MK-Baire} and Corollary \ref{c:lc-kFU-Baire} essentially generalize Theorem \ref{t:m-K-Baire} and  Theorem  \ref{t:sc+FU=>Baire}, respectively.

In Section \ref{exa:feral-Baire} we show that the property $(MK)$ (and hence the property $(K)$), $\kappa$-Fr\'{e}chet--Urysohness and Baireness are completely independent. Taking into consideration Theorem \ref{t:MK-Baire}, the most intriguing and interesting problem is the following one: {\em Is it true that if an lcs $E$ is Baire, then it is a $\kappa$-Fr\'{e}chet--Urysohn space}? This problem is especially interesting because for the classical function spaces, the $\kappa$-Fr\'{e}chet--Urysohness is a very natural necessary condition to be a Baire space.
Indeed, let $X$ be a Tychonoff space. We denote by $C_p(X)$ and $\CC(X)$ the space $C(X)$ of all real-valued functions on $X$ endowed with the pointwise topology or the compact-open topology, respectively. In \cite[Proposition~2.6]{Sak2}, Sakai proved that if $C_p(X)$ is a Baire space, then it is $\kappa$-Fr\'{e}chet--Urysohn. Analogously, Sakai \cite[Proposition~2.5]{Sakai} proved that also the Baireness of $\CC(X)$ implies that $\CC(X)$ is a $\kappa$-Fr\'{e}chet--Urysohn space.
These results motivate the following problem: Is it true that if an lcs $E$ is Baire, then it is a $\kappa$-Fr\'{e}chet--Urysohn space? In Example  \ref{lcb-not_kFU}  we answer this problem negatively by showing that there is a Baire lcs $E$ which is even feral  but not $\kappa$-Fr\'{e}chet--Urysohn.

Although $\kappa$-Fr\'{e}chet--Urysohn spaces can be not Baire (see Example \ref{exa:metr-not-K-Baire} below), it turns out that any $\kappa$-Fr\'{e}chet--Urysohn lcs $E$ has a weaker Baire-type property. In Theorem \ref{t:kFU=>b-Baire-like} we prove that each $\kappa$-Fr\'{e}chet--Urysohn locally convex space $E$ is $b$-Baire-like, and hence $E$ is quasibarrelled (for all relevant definitions see Section \ref{sec:kFU-Baire}). This somewhat surprising result gives numerous new examples of $b$-Baire-like spaces even in $\CC$-spaces, see Corollary \ref{c:Ck-b-Baire-like} below. 
In Theorem \ref{t:kFU-Baire-like} we characterize $\kappa$-Fr\'{e}chet--Urysohn locally convex spaces which are Baire-like.

In the last Section \ref{sec:B1}, we give some applications of Theorem \ref{t:MK-Baire} and its Corollary \ref{c:lc-kFU-Baire}. In particular, we give  very short proofs of two highly non-trivial results about Baireness of the space $B_\alpha(X)$ of Baire-$\alpha$ functions.
We  prove that for a $k_\IR$-space $X$, $\CC(X)$ is a Baire space if, and only if, $\CC(X)$ is a $\kappa$-Fr\'{e}chet--Urysohn space
if, and only if, $X$ has the property $(\kappa_k)$ of Sakai.


\section{Property $(K)$ and property $(MK)$ of topological vector spaces  } \label{sec:K-MK}


In this section we study the property $(K)$ and the property $(MK)$ and relationships between them. In the next assertion we show that the property $(MK)$ is weaker than the property $(K)$.

\begin{proposition} \label{p:K=>MK}
If a tvs $E$ has the property $(K)$, then it has the property $(MK)$.
\end{proposition}

\begin{proof}
To prove the proposition it suffices to show that any Mackey null sequence $\{x_n\}_{n\in\w}$ in $E$ is a null sequence. To this end, let $U$ be an arbitrary absorbent and balanced neighborhood of $0$, and let $\{a_n\}_{n\in\w}$ be an increasing unbounded sequence of positive numbers such that $a_n x_n \to 0$ in $E$. We can assume that $a_n\geq 1$. Then there exists $N$ such that $a_n x_n\in U$ for any $n>N$. Since $U$ is balanced and
$a_n\geq 1$, we obtain that $x_n=\tfrac{1}{a_n}\cdot a_nx_n\in U$ for any $n>N$. Thus $x_n\to 0$, as desired.
\end{proof}

For locally convex spaces the property $(MK)$ is weaker than local completeness.

\begin{proposition} \label{p:lc-k-complete}
A locally complete lcs $E$ has the property $(MK)$.
\end{proposition}

\begin{proof}
Let $S=\{x_n\}_{n\in\w}\subseteq E$ be a Mackey null sequence. To prove that $E$  has the property $(MK)$, we show that $S$ is a $K$-sequence. Let $\{z_i\}_{i\in\w} $ be an arbitrary subsequence of $S$. As $S$ is Mackey null, there exists an  increasing unbounded sequence  $\{a_n\}_{n\in\w}$ of positive numbers such that $a_nx_n \to 0$ in $E$. Therefore there is an increasing sequence  $(n_k)_{k\in\w}$ in $\w$ such that $a_{n_k}\geq 2^{k+1}$ and $\{x_{n_k}\}_{k\in\w}$ is a subsequence of $\{z_i\}_{i\in\w}$. It suffices to show that the series $\sum_{k\in\w} x_{n_k}$ converges in $E$.

Set $y_k:=a_{n_k} x_{n_k}$ and $\mu_k:=\tfrac{1}{a_{n_k}}$ for $k\in \w$. Then $x_{n_k} = \mu_k y_k$ and $|\mu_k|\leq \tfrac{1}{2^{k+1}}$ for $k\in \w$.
Let $K$ be the closed absolutely convex hull of $\{y_k\}_k$. Since $y_k\to 0$ and $E$ is a locally complete lcs, \cite[Theorem 2.10.21]{BS2017tvs} implies that $K$ is a compact absolutely convex subset of $E$.
Since $y_k\in K$ and $|\mu_k|\leq \tfrac{1}{2^{k+1}}$ for $k\in \w$, the sequence $\{\sum_{i\leq k} \mu_k y_k\}_{k\in\w}$ is fundamental in the Banach space $E_K$ (where $E_K$ denotes the span of $K$ endowed with the norm topology defined by the Minkowski functional of $K$). Therefore the series $\sum_{k\in\w} \mu_k y_k$ converges in $E_K$ and hence also in $E$.
\end{proof}

In the proof of Proposition \ref{p:lc-k-complete} we essentially used the local convexity of $E$.
\begin{problem}
Is it true that a (sequentially) complete tvs has the property $(MK)$?
\end{problem}

Using Proposition \ref{p:lc-k-complete} we show that the converse in Proposition \ref{p:K=>MK} is not true in general.

\begin{example} \label{exa:MK-not=K}
The complete locally convex space  $\IR^{\ell_2}$ has the property $(MK)$ but it does not have the property $(K)$.
\end{example}

\begin{proof}
Since $\IR^{\ell_2}$ is complete, Proposition \ref{p:lc-k-complete} implies that it has the property $(MK)$. On the other hand, consider the sequence $\{x_n\}_{n\in\w}$ in $\IR^{\ell_2}$ defined by
\[
x_n=\big(\xi(n)\big)_{\xi\in\ell_2} \quad (n\in\w).
\]
It is clear that it converges to zero in $\IR^{\ell_2}$. Let $\{x_{n_k} \}_{k\in\w}$ be an arbitrary subsequene of $\{x_n\}_{n\in\w}$. Take $\xi\in\ell_2$ such that
\[
\xi(n) = \begin{cases}
\frac{1}{k+1}, & \mbox{if } n = n_k\text{ for some }k\in\w,
\\
0, & \text{othewise}.
\end{cases}
\]
The choice of $\xi$ implies that the $\xi$-coordinate of the series $\sum_k x_{n_k}$ is infinite. Therefore $\sum_k x_{n_k}$ diverges. Thus $\IR^{\ell_2}$  does not have the property $(K)$.
\end{proof}

Proposition \ref{p:K=>MK} motivates the problem to find classes of topological vector spaces for which the property $(K)$ is equivalent to the property $(MK)$. If $E$ is a metrizable lcs, then every null sequence in $E$ is Mackey null \cite[Proposition 5.1.4]{PB}. Therefore, a metrizable lcs has the property $(K)$ if, and only if, it has the property $(MK)$. We shall significantly strengthen this statement in Proposition \ref{p:al4-MK=K} below.

Recall that a Tychonoff space $X$ has the {\em property $(\alpha_4)$} at a point $x\in X$  if for every family of sequences $\{\xi_n\}_{n\in\w}$ converging to $x$, there exists a sequence $\xi$ converging to $x$ such that $\xi_n\cap\xi\neq\emptyset$ for infinitely many $n\in\w$.
Following Arhangel'skii,  $X$ has the {\em property $(\alpha_4)$} if it has the property $(\alpha_4)$ at each point $x\in X$.
Each Fr\'{e}chet--Urysohn topological group has the property $(\alpha_4)$, see \cite[Theorem 4]{Nyikos1981}.

\begin{proposition} \label{p:al4-MK=K}
Let $E$ be a topological vector space with the property $(\alpha_4)$ {\rm(}for example, $E$ is Fr\'{e}chet--Urysohn{\rm)}. Then $E$ has the property $(K)$ if, and only if, it has the property $(MK)$.
\end{proposition}

\begin{proof}
Taking into account Proposition \ref{p:K=>MK}, we have to prove only that if $E$ has the property $(MK)$, then it has  the property $(K)$. To this end, it suffices to prove that each null sequence $\xi=\{x_n\}_{n\in\w}$ contains a Mackey null subsequence. For every $k\in\w$, set $\xi_k=\{ 2^k x_n\}_{n\in \w}$. The property $(\alpha_4)$ implies that there are sequences $(k_m)_m$ and $(n_m)_m$ in $\w$ such that $\{2^{k_m} x_{n_m}\}_{m\in \w}$ is a null sequence. Then $\{x_{n_m}\}_{m\in \w}$ is a Mackey null sequence.
\end{proof}

The following assertion follows from Propositions \ref{p:lc-k-complete} and \ref{p:al4-MK=K}.
\begin{corollary} \label{c:FU+sc=>K}
Every locally {\rm(}in particular, sequentially{\rm)} complete Fr\'{e}chet--Urysohn lcs have the property $(K)$.
\end{corollary}

\begin{problem}\label{prob:FU+sc=>K}
Is it true that every sequentially complete Fr\'{e}chet--Urysohn topological vector space has the property $(K)$?
\end{problem}
In the special important case when the space is metrizable, the answer to Problem \ref{prob:FU+sc=>K} is positive (see \cite[Observation 1.2.16(a)]{PB}).
\begin{proposition}[\cite{PB}] \label{p:metr+c=>K}
Every  complete metrizable tvs has the property $(K)$.
\end{proposition}

Recall that a tvs $E$ is {\em quasi-complete} if any closed bounded subset of $E$ is complete. Each complete tvs is quasi-complete,  each quasi-complete tvs is sequentially complete, and each sequentially complete lcs is locally complete, see \cite{Jar}. Recall also that a tvs $E$ is called {\em feral} if every bounded subset of $E$ is finite-dimensional.

A subset $A$ of a Tychonoff space $X$ is {\em functionally bounded} if $f(A)$ is bounded for every $f\in C(X)$.
Recall that  $X$ is a {\em $\mu$-space} if any functionally bounded subset of $X$ is relatively compact.

Below we give a sufficient condition on a tvs to have the property $(K)$.

\begin{proposition} \label{p:feral-K}
Let $E$ be a feral topological vector space. Then
\begin{enumerate}
\item[{\rm (i)}] $E$ is quasi-complete {\rm(}hence sequentially complete{\rm)};
\item[{\rm (ii)}] $E$ is a $\mu$-space;
\item[{\rm (iii)}] $E$ has  the property $(K)$.
\end{enumerate}
\end{proposition}

\begin{proof}
(i) follows from the fact that a closed bounded subset of $E$ is even compact.

(ii) follows from the fact that any functionally bounded subset of $E$ is bounded.

(iii) Let $\xi=\{x_n\}_{n\in\w}$ be a null sequence in $E$. Since $\xi$ is bounded and $E$ is feral, $\xi$ lies in some finite-dimensional subspace $L\subseteq E$. By Proposition \ref{p:metr+c=>K}, the subspace $L$ has property $(K)$. Therefore $\xi$ is a $K$-sequence.
\end{proof}

The following statement shows that it is consistent  that the answer to Problem \ref{prob:FU+sc=>K} is positive.
\begin{theorem} \label{t:FU_sc_tvs=>K}
It is consistent that each sequentially complete Fr\'{e}chet--Urysohn tvs $E$ has the property $(K)$.
\end{theorem}

\begin{proof}
It is known (see \cite{HrusakRamos-Garcia2014}) that there is a $\mathrm{ZFC}$ model in which each Fr\'{e}chet--Urysohn separable topological group is metrizable. Let us show that $E$ has the property $(K)$. Let $\{x_n\}_{n\in\w}$ be a null sequence in $E$, and let $L\subseteq E$ be a closed separable linear subspace that contains the sequence $\{x_n\}_{n\in\w}$. Then $L$ is metrizable. Since $L$ is sequentially complete and metrizable, it follows that $L$ is complete (\cite[Proposition~3.2.2]{Jar}). Therefore, by Proposition \ref{p:metr+c=>K},  $L$ has the property $(K)$. Thus, $\{x_n\}_{n\in\w}$ is a $K$-sequence.
\end{proof}

Example \ref{exa:MK-not=K} shows that there are locally complete (even complete) locally convex spaces without the property $(K)$. On the other hand, in \cite{La-Li} it is shown that there are non-complete metrizable locally convex spaces with the property $(K)$.
Taking into account that a metrizable lcs is locally complete if, and only if, it is complete (\cite[Corollary~5.1.9]{PB}), the obtained results show that, in the realm of all locally convex spaces,  the  following diagram holds (in which $FU$= Fr\'{e}chet--Urysohn):

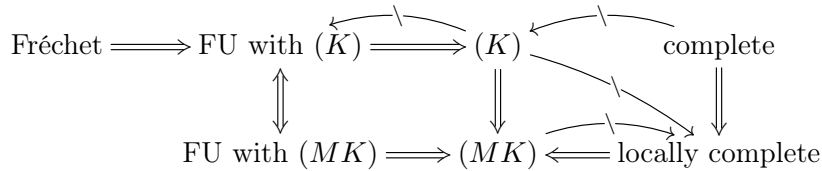
\[
\xymatrix{
\mbox{Fr\'{e}chet}  \ar@{=>}[r] & \mbox{FU with $(K)$}  \ar@{=>}[r] \ar@{<=>}[d]  & \mbox{$(K)$} \ar@{=>}[d] \ar@/^/[rd]|-{\setminus} \ar@/_1pc/[l]|-{\setminus} & \mbox{complete}  \ar@{=>}[d] \ar@/_1pc/[l]|-{\setminus}  \\
& \mbox{FU with $(MK)$}  \ar@{=>}[r]  & \mbox{$(MK)$}\ar@/^1pc/[r]|-{\setminus} & \mbox{locally complete} \ar@{=>}[l]  }
\]

In the next assertion we show that spaces with the property $(MK)$ are (closely) hereditary. Recall that a subspace $Y$ of a Tychonoff space $X$ is {\em sequentially closed} if for every sequence $\{y_n\}_{n\in\w}\subseteq Y$ converging to some $x\in X$, it follows that $x\in Y$.

\begin{proposition} \label{p:MK-subspace}
Let $L$ be a sequentially closed linear subspace of a topological vector space $E$. If $E$ has the property $(MK)$, then also $L$ has  the property $(MK)$.
\end{proposition}

\begin{proof}
Let $S=\{x_n\}_{n\in\w}$ be a Mackey null sequence in $L$. Then $S$ is a Mackey null sequence in $E$. Therefore $S$ is a $K$-sequence in $E$. To show that $S$ is a $K$-sequence in $L$, let $(y_n)$ be a subsequence of $S$. Choose a subsequence $(z_k)$ of $(y_k)$ such that the series $\sum_k z_k$ is convergent in $E$. Since all $z_k$ are in $L$ and $L$ is sequentially closed, it follows that $\sum_k z_k$ is convergent in $L$. Thus $S$ is a $K$-sequence in $L$, and hence $L$ has  the property $(MK)$.
\end{proof}

In the proof of the main results of the article (Theorem \ref{t:kFU-MK-dense-Baire}) we need the following two more specific notions.

\begin{definition}\label{def:rMK-property}
A subset $A$ of a tvs $E$ is said to have a {\em relative property $(MK)$} if every Mackey null sequence in $A$ is a $K$-sequence in $E$.
\end{definition}

In particular, if $A=E$, we obtain the property $(MK)$ defined in Definition \ref{def:MK-property}.

\begin{definition}
A subspace $S$ of a uniform space $X$ is said to be {\em relatively sequentially complete} if every Cauchy sequence in $S$ converges to an element of $X$.
\end{definition}

\begin{proposition}\label{p:rsc-rksc}
A relatively sequentially complete  subset $L$ of a locally convex space $E$ has the relative property $(MK)$.
\end{proposition}

\begin{proof}
Let $\widehat E$ be a completion of $E$. Let $(x_n)_{n}\subseteq L$ be a Mackey null sequence. Then $(x_n)_{n}$ is Mackey null also in $\widehat E$. Since, by Proposition \ref{p:lc-k-complete},  $\widehat E$ has the property  $(MK)$, $(x_n)_{n}$ is a $K$-sequence in $\widehat E$. To show that $(x_n)_{n}$ is a $K$-sequence also in $E$, let $(z_m)_m$ be an arbitrary subsequence of  $(x_n)_{n}$. Then there is a subsequence  $(y_n)_{n}$ of  $(z_m)_{m}$ such that the series $\sum_n y_n$ converges in $\widehat E$. Since $L$ is a relatively sequentially complete subset of $E$ and the sequence $\big\{\sum_{i\leq n} y_i\big\}_{n\in\w}\subseteq L$ is Cauchy in $E$, the series $\sum_n y_n$ converges in $E$. Thus $(x_n)_{n}$ is a $K$-sequence in $E$.
\end{proof}


\section{Baire property of $\kappa$-Fr\'{e}chet--Urysohn topological vector spaces  } \label{sec:kFU-Baire1}


We start this section with the next characterization of tvs which are $\kappa$-Fr\'{e}chet--Urysohn.

\begin{theorem} \label{t:kFU-group-charac}
For a topological vector space $E$, the following assertions are equivalent:
\begin{enumerate}
\item[{\rm(i)}] for every sequence $\{U_n\}_{n\in\w}$ of open subsets of $E$ such that $0\in \overline{U_n}$ for every $n\in\w$, there are a strictly increasing sequence $(n_k)\subseteq \w$ and a sequence $\{x_k\}_{k\in\w}$ in $E$ such that $x_k\in U_{n_k}$ $(k\in\w)$ and $x_k\to 0$;
\item[{\rm(ii)}] $E$ is a $\kappa$-Fr\'{e}chet--Urysohn space.
\end{enumerate}
\end{theorem}

\begin{proof}
(i)$\Ra$(ii) Let $U$ be an open subset of $E$ such that $0\in \overline{U}\SM U$. For every $n\in\w$, set $U_n:=U$. Then (ii) immediately follows from (i).
\smallskip

(ii)$\Ra$(i) Let $\{U_n: n\in \w\}$  be a sequence of open subsets of $E$ such that $0\in \overline{U_n}$ for every $n\in\w$. Choose an arbitrary one-to-one sequence $\{p_n\}_{n\in\w}$ in $E\SM \{0\}$  converging to $0$ (such a sequence exists by the $\kappa$-Fr\'{e}chet--Urysohness applying to the open set $E\SM\{0\}$).  Choose a sequence $\{W_n: n\in \w\}$ of open neighborhoods of $0$ such that
\begin{equation} \label{equ:kFU-group-charac-1}
0\not\in \overline{p_n +W_n}
\end{equation}
for every $n\in \w$. For each $n\in\w$, set
\begin{equation} \label{equ:kFU-group-charac-2}
\mathcal{O}_n=(p_n+U_n)\cap (p_n+W_n).
\end{equation}
We claim that $0\in \overline{\bigcup_{n\in\w} \mathcal{O}_n}$.  Indeed, let $V$ be a neighborhood of $0$. Choose a neighborhood $V_1$ of $0$ such that $V_1+ V_1\subseteq U$. Since $p_n\to 0$, there is an $m\in\w$ such that $p_n\in V_1$ for each $n\geq m$.  As $0\in \overline{U_m}$, there is a point $s_m\in U_m$ such that $s_m\in V_1\cap W_m$. Then $p_m+ s_m\in V_1+ V_1\subseteq V$ and $p_m+ s_m\in \mathcal{O}_m$.
\smallskip

Since $E$ is $\kappa$-Fr\'{e}chet--Urysohn, there is a sequence $(z_k)\subseteq \bigcup \{\mathcal{O}_n: n\in\w\}$ converging to $0$. For every $k\in\w$, choose an index $n_k$ such that $z_k\in \mathcal{O}_{n_k}$. The sequence $(n_k)$ cannot be bounded because of (\ref{equ:kFU-group-charac-1}) and the inclusion $\mathcal{O}_{n_k} \subseteq p_{n_k}+W_{n_k}$. Passing to a subsequence if needed, we can assume that $(n_k)$ is strictly increasing.  For every $k\in\w$, set $x_k:= z_k -p_{n_k} $. Then, by (\ref{equ:kFU-group-charac-2}), $x_k\in U_{n_k}$ and, clearly, $x_k\to 0$. Thus (i) is satisfied.
\end{proof}

The following theorem is the main result of this paper, it immediately implies Theorem \ref{t:MK-Baire}.

\begin{theorem} \label{t:kFU-MK-dense-Baire}
If a topological vector space $E$ contains a dense $\kappa$-Fr\'{e}chet--Urysohn linear subspace $L$ with the relative property $(MK)$, then $E$ is a Baire space.
\end{theorem}

\begin{proof}
Suppose for a contradiction that $E$ is not a Baire space. Then, by Theorem \ref{t:Saxon}, 
there exists an absorbing closed and balanced set $B$ in $E$ with the empty interior.
\smallskip

{\em Claim 1. There exists a sequence $\{x_k\}_{k\in\w}$ in $L$ such that $x_k \to 0$ and $x_k\notin k2^k B$.} Indeed, for every $k\in\w$, set $U_n:=L\setminus k2^k B$. Since $B$ is nowhere dense in $E$, we have $0\in \overline{U_n}$. As $L$ is $\kappa$-Fr\'{e}chet--Urysohn, Theorem \ref{t:kFU-group-charac} implies that there exist a sequence $\{x_k\}_{k\in\w}$ in $L$ and a sequence $(n_k)_k\subseteq \w$ such that $x_k\to 0$ and $x_k\in U_{n_k}$ for each $k\in \w$. We can assume that $n_k\geq k$. Then $x_k\notin k2^k B$. Claim 1 is proved.
\smallskip

For every $n\in\w$, set $y_n:=\tfrac{1}{2^n} x_n\in L$. Clearly, $y_n\notin n B$ and the sequence $(y_n)_n$ is Mackey null. Since $L$ has the relative property $(MK)$, the sequence $(y_n)_n$ is a $K$-sequence in $E$. Applying Lemma 14.3 of \cite{kak} to the sequence $(kB)$ and the $K$-sequence $(y_n)_n$ we can find a finite subset $M$ of $E$ and a number $m\in\w$ such that $\{y_n\}_{n\in\w} \subseteq mB+M$. Since $B$ is absorbing and $M$ is finite, we obtain that $\{y_n\}_{n\in\w} \subseteq tB$ for some $t\in\w$. Therefore, if $k>t$, we obtain $x_k=2^k y_k\in k 2^k B$. But this contradicts the choice of the sequence $\{x_k\}_{k\in\w}$.
\end{proof}

We select the following corollary of Theorem \ref{t:MK-Baire} and Proposition \ref{p:K=>MK}.

\begin{corollary} \label{c:K+FU=>Baire}
Each Fr\'{e}chet--Urysohn tvs $E$ with the property $(K)$ is a Baire space.
\end{corollary}


The following result is classical (for the proof see, for example, \cite[Theorem~14.2]{kak}), we provide its completely independent proof.
\begin{corollary} \label{c:product-Frechet-Baire}
The product of a family of Fr\'{e}chet spaces is a Baire space. In particular, $\IR^\lambda$ is a Baire space for any cardinal $\lambda$.
\end{corollary}

\begin{proof}
By \cite[Theorem~3.1]{Mrowka}, the product of any family of metrizable topological spaces is a $\kappa$-Fr\'{e}chet--Urysohn space. On the other hand, the product of any family of complete topological vector spaces is complete, see Proposition 3.2.6 of \cite{Jar}. Now Corollary \ref{c:lc-kFU-Baire} applies.
\end{proof}

\begin{problem}
Let $E$ be a  {\rm(}sequentially{\rm)} complete  {\rm(}$\kappa$-{\rm)}Fr\'{e}chet--Urysohn tvs. Is $E$ Baire?
\end{problem}

The following theorem provides a partial answer to this question.

\begin{theorem} \label{t:FU_sc_tvs-Baire}
It is consistent that each sequentially complete Fr\'{e}chet--Urysohn tvs $E$ is a Baire space.
\end{theorem}
\begin{proof}
By Theorem \ref{t:FU_sc_tvs=>K}, it is consistent that each sequentially complete Fr\'{e}chet--Urysohn tvs $E$ has the property $(K)$. Now Corollary \ref{c:K+FU=>Baire} 
implies that $E$ is a Baire space.
\end{proof}

The following statement follows from Theorem \ref{t:kFU-MK-dense-Baire} and Proposition \ref{p:rsc-rksc}.

\begin{theorem} \label{t:kFU_rsc-Baire}
If a locally convex space $E$ contains a relatively sequentially complete  $\kappa$-Fr\'{e}chet--Urysohn dense linear subspace, then $E$ is a Baire space.
\end{theorem}

\begin{corollary} \label{c:kFU_complete-Baire}
The completion $\widehat E$ of a $\kappa$-Fr\'{e}chet--Urysohn locally convex space $E$ is a Baire space.
\end{corollary}


\section{ Independence of  the property $(MK)$, $\kappa$-Fr\'{e}chet--Urysohness and Baireness}\label{exa:feral-Baire}



In this section we show that the property $(MK)$, $\kappa$-Fr\'{e}chet--Urysohness and Baireness are completely independent. We start from the next simple example which shows that $\kappa$-Fr\'{e}chet--Urysohness implies neither Baireness nor the property $(K)$.

\begin{example} \label{exa:metr-not-K-Baire}
The subspace $E$ of $\IR^\w$ of all vectors with finite support is trivially $\kappa$-Fr\'{e}chet--Urysohn but not Baire and hence, by Theorem \ref{t:m-K-Baire}, $E$ does not have the property $(MK)$.
\end{example}

The following example shows that the property $(K)$ implies neither $\kappa$-Fr\'{e}chet--Urysohness nor  Baireness.

\begin{example} \label{exa:K-not-Baire-kFU}
The direct locally convex sum $\varphi$ of a countably infinite family of $\IR$-s has the property $(K)$, but it is  neither $\kappa$-Fr\'{e}chet--Urysohn nor  Baire.
\end{example}

\begin{proof}
Since $\varphi$ is feral (see \cite[Proposition~4.3.4]{Jar}), it has the property $(K)$ by Proposition \ref{p:feral-K}. It is clear that $\varphi$ is not a Baire space. Finally, $\varphi$ is not a  $\kappa$-Fr\'{e}chet--Urysohn space by Proposition 2.9 of \cite{Gabr-B1}.
\end{proof}

By Corollary \ref{c:lc-kFU-Baire}, a locally complete $\kappa$-Fr\'{e}chet--Urysohn locally convex space is Baire. One can naturally ask whether the converse is true, that is: {\em Is it true that a locally complete Baire space is $\kappa$-Fr\'{e}chet--Urysohn}? We answer this problem in the negative. First we prove several auxiliary results.

\begin{lemma}\label{l:Rn+1-ls}
For every $n\in\NN$, there exists a sequence $(V_m)_{m\in \w}$ of open subsets of $\IR^{n+1}$ such that
\begin{enumerate}
\item[{\rm(i)}] for every $m\in\w$, $V_m$ is bounded and $0\not\in \overline{V_m}$;
\item[{\rm(ii)}] $(V_m)_{m\in \w}$ converges to zero, i.e., for every neighborhood $W$ of $0$ there is $k\in\w$ such that $V_m\subseteq W$ for every $m\geq k$;
\item[{\rm(iii)}] $|\{m\in\w: V_m\cap L\neq\emptyset\}|\leq n $ for any linear subspace $L\subseteq \IR^{n+1}$ of dimension $n$.
\end{enumerate}
\end{lemma}

\begin{proof}
Denote by $\IR_n[x]$ the linear space of all polynomials of degree at most $n$ over the field $\IR$.
We shall identify $\IR^{n+1}$ with the dual space $\IR_n'[x]$ of $\IR_n[x]$. With this identification and taking into account that any $n$-dimensional linear subspace of $\IR_n'[x]$ is uniquely defined by a nonzero polynomial $f\in \IR_n[x]$ (as the kernel of $f$ under the duality $f(l)=l(f)$ for $f\in \IR_n[x]$ and $l\in \IR_n'[x]$), to prove the lemma we first construct a sequence $(U_m)_{m\in \w}$ of open {\em bounded} subsets of $\IR_n'[x]$ such that $0\not\in \overline{U_m}$ for every $m\in\w$, and
\begin{equation}\label{eq:lemRn+1}
|\{m\in\w: 0\in f(U_m) \}|\leq n \; \text{ for each }\; f\in \IR_n[x]\setminus \{0\}.
\end{equation}

For $t\in\IR$, let $\delta_t\in \IR_n'[x]$ denote the Dirac functional: $\delta_t(f)=f(t)$ for $f\in \RR_n[x]$.
\smallskip

{\em {\bf Claim 1.} If $t_0,t_1,...,t_m$ are distinct elements of $\IR$ and $m\leq n$, then the functionals $\delta_{t_0},\delta_{t_1},...,\delta_{t_m}$ are linearly independent.}
\smallskip

{\em Proof of Claim 1.} By adding distinct elements $t_{m+1},\dots,t_n$, without loss of generality we assume that $m=n$. For every $i=0,1,...,n$, let $f_i\in \IR_n[x]$ be such that $f_i(t_j)=0$ if $j\neq i$ and $f_i(t_j)=1$ if $j= i$. Then the matrix $\big(\delta_{t_i}(f_j)\big)_{ij}=\big(f_j(t_i)\big)_{ij}$ is the identity matrix and, therefore, nonsingular. Consequently, the functionals $\delta_{t_0},\delta_{t_1},...,\delta_{t_m}$ are linearly independent. Claim 1 is proved.
\smallskip

We set
\[
\e := \tfrac{1}{2(n+1)2^n} \; \text{ and } \; q_m := (n+1) 2^m \; \text{ for }\; m\in\w.
\]
\smallskip

{\em {\bf Claim 2.} Let $m_0,m_1,...,m_n\in\w$ be distinct numbers, and let $f\in \IR_n[x]$ be such that $|f(q_{m_i})|<\e$ for $i=0,1,...,n$.  Then $|f(k)|\leq \tfrac{1}{2}$ for every $k=0,1,...,n$.}
\smallskip

{\em Proof of Claim 2.} To prove the claim, for every $i=0,1,...,n$, we set $y_i:=f(q_{m_i})$ and
\[
f_i(x) =\frac
{(q_{m_0}-x)(q_{m_1}-x)\cdots(q_{m_{i-1}}-x)(q_{m_{i+1}}-x)\cdots(q_{m_{n}}-x)}
{(q_{m_0}-q_{m_i})(q_{m_1}-q_{m_i})\cdots(q_{m_{i-1}}-q_{m_i})(q_{m_{i+1}}-q_{m_i})\cdots(q_{m_{n}}-q_{m_i})}.
\]
Then
\[
f(x)=\sum_{i=0}^n y_i f_i(x).
\]
By the assumption of the claim we have $|y_i|<\e$. Since $n < q_{m_j}$ for each $j=0,1,...,n$, we obtain
\[
|f_i(k)|\leq |f_i(0)| \;\; \mbox{ for every $k=0,1,...,n$}.
\]

Let us estimate $|f_i(0)|$. Since
\[
f_i(0) =\frac{1}
{(1-\frac{q_{m_i}}{q_{m_0}})(1-\frac{q_{m_i}}{q_{m_1}})\cdots(1-\frac{q_{m_i}}{ q_{m_{i-1}}})(1-\frac{q_{m_i}}{q_{m_{i+1}}})\cdots(1-\frac{q_{m_i}}{q_{m_{n}}})}
\]
and $|1-\tfrac{q_{m_i}}{q_{m_j}}|=|1-2^{m_i-m_j}|\geq \frac{1}{2}$ for distinct $i$ and $j$, we obtain
\[
|f_i(0)|\leq 2^{n}.
\]
Therefore, for every $k=0,1,...,n$, we have
\[
|f(k)| = \Big|\, \sum_{i=0}^n y_i f_i(k) \,\Big| \leq \sum_{i=0}^n |y_i| \cdot|f_i(k)| \leq \sum_{i=0}^n |y_i|\cdot |f_i(0)| < \e \cdot (n+1) 2^n
= \tfrac{1}{2}.
\]
Claim 2 is proved.
\smallskip

Now we define
\begin{align*}
V  &:= \big\{ f\in \IR_n[x] : |f(\delta_k)|<1 \text{ for }k=0,1,...,n\big\},
\\
U &:= \big\{ l\in \IR_n'[x] : l(V) \subseteq (-1,1) \big\}.
\end{align*}
Since, by Claim 1,  $\delta_0,\dots,\delta_n$ are linearly independent in the $(n+1)$-dimensional space $\IR'_n[x]$, it follows that $V$ and hence also $U$ are open bounded neighborhoods of zero. For every $m\in\w$, set
\[
U_m := \delta_{q_m} + \e U.
\]
Then all $U_m$ are open bounded subsets of $\IR_n'[x]$. We claim that $0\not\in\overline{U_m}$. Indeed, choose a polynomial $g\in V$ such that $g(q_m)>2$. Then, for every $l\in U$, we have
\[
g(\delta_{q_m} + \e l)\geq g(q_m)-\e>1,
\]
which implies that $0\not\in\overline{U_m}$.

Let us check that the sequence $(U_m)_{m\in\om}$ satisfies condition (\ref{eq:lemRn+1}).
Assuming the converse we could find distinct numbers $m_0,m_1,...,m_n\in\w$ and $g\in \IR_n[x]\setminus\{0\}$ such that $0\in g(U_{m_i})$ for $i=0,1,...,n$.
Let $s=\max \{|g(k)|:k=0,1,...,n\}$. Since $g\neq 0$, we have $s>0$. Set $f(x):= \tfrac{2g(x)}{3s}$.
Then $f\in V$ and $|f(k_0)|>\tfrac{1}{2}$ for some $k_0\in \{0,1,...,n\}$.

Let $i\in\{0,1,...,n\}$. Then $0\in f(U_{m_i})$. Choose $h_i\in U$ such that $f(\delta_{q_{m_i}}-\e h_i)=0$. Then $f(q_{m_i})=\e f(h_i)$. Since $f\in V$ and $h_i\in U$, we have $|f(h_i)|< 1$. Therefore, $|f(q_{m_i})|<\e$. Applying Claim 2 to the polynomial $f$, we obtain  that $|f(k_0)|\leq \tfrac{1}{2}$. This is a contradiction.
\smallskip

To finish the proof of the lemma, for every $m\in\w$,  it suffices to set $V_m:= a_m U_m = a_m(\delta_{q_m} + \e U)$, where $a_m\to 0$ sufficiently fast. Then the properties of $U_m$-s proved above imply (i)-(iii) for the sequence  $(V_m)_{m\in \w}$, as desired.
\end{proof}

\begin{proposition} \label{p:feral-not-kFU}
Let $A$ be an infinite set, and let $E\subseteq \IR^A$ be a dense linear subspace. If $E$ is feral, then $E$ is not $\kappa$-Fr\'{e}chet--Urysohn.
\end{proposition}
\begin{proof}
Suppose for a contradiction that $E$ is a $\kappa$-Fr\'{e}chet--Urysohn space.
\smallskip

Let $(M_n)_n$ be a pairwise disjoint sequence of finite subsets of $A$ such that $|M_n|=n+1$. For every $n\in\NN$, by Lemma \ref{l:Rn+1-ls}, there exists a sequence $(V_{n,k})_{k\in\NN}$ of open subsets of $\IR^{M_n}$ such that
\begin{enumerate}
\item[$(\alpha)$] $0\notin \overline{V_{n,k}}$ for every $k\in\NN$,
\item[$(\beta)$] $(V_{n,k})_k$ converges to $0$, and
\item[$(\gamma)$] $|\{k\in\NN: V_{n,k}\cap L\neq\emptyset\}|\leq n$ for any $n$-dimensional linear subspace $L\subseteq \IR^{M_n}$.
\end{enumerate}

For every  $n\in \NN$, set
\[
W_n := \bigcap_{i=1}^n \pi^{-1}_{M_i}(V_{i,n})\;\; \mbox{ and }\;\; U_n := E \cap W_n
\]
and put $U:=\bigcup_{n=1}^\infty U_n$. The density of $E$ in $\IR^A$ implies that all $U_n$ are nonempty.
\smallskip

(a) We claim that $0\notin\overline{U_n}$ for every  $n\in\NN$. Indeed, otherwise, we would have $\pi_{M_n}(0)\in \overline{\pi_{M_n}(U_n)}\subseteq \overline{V_{n,n}}$ that is impossible by the property $(\alpha)$ of the sequence $(V_{n,k})_{k\in\NN}$.
\smallskip

(b) We show that $0\in \overline{U}$. Since $U$ is dense in $W := \bigcup_{n=1}^\infty W_n$, it suffices to verify that $0\in \overline{W}$.
As $(M_n)_n$ is pairwise disjoint , for every $k\in \NN$, there is $v_k\in \IR^A$ such that $\pi_{M_n}(v_k)\in V_{n,k}$ for $n\in \NN$ and $\pi_\alpha(v_k)=0$ for $\alpha\in A\setminus \bigcup_{n=1}^\infty M_n$. Then $v_k\in W_k$ for all $k\in \NN$. It follows from $(\beta)$ that the sequence $(v_k)_k$ converges to $0$. Consequently, $0\in \overline{W}$.
\smallskip

(c) Let us show that $\{n\in\NN: U_{n}\cap L\neq\emptyset\}$  is finite for any finite-dimensional linear subspace $L\subseteq E$. Indeed, let $m$ be the dimension of $L$. By $(\gamma)$, choose $n_0>m$ such that $V_{m,n}\cap \pi_{M_m}(L)=\emptyset$ for every $n\geq n_0$. Then, for every $n\geq n_0$, we obtain
\[
\pi_{M_m}(U_n \cap L)\subseteq \pi_{M_m}(U_n) \cap \pi_{M_m}(L)\subseteq V_{m,n}\cap \pi_{M_m}(L)=\emptyset,
\]
and hence $U_n \cap L=\emptyset$. Thus the set $\{n\in\NN: U_{n}\cap L\neq\emptyset\}$ is finite, as desired.
\smallskip

By our supposition, $E$ is a $\kappa$-Fr\'{e}chet--Urysohn space. Therefore, by (b), there exists a null sequence $(x_k)_k\subseteq U$. Then $x_k\in U_{n_k}$ for some $n_k\in \NN$. Passing to a subsequence if needed and by (a),  we can assume that $(n_k)_k$ is a strictly increasing sequence.
Since $E$ is feral, the sequence  $(x_{k})_k$ is contained in some finite-dimensional subspace $L\subseteq E$. Applying (c) we see that the set $S:=\{n\in\NN: U_{n}\cap L\neq\emptyset\}$ is finite. Consequently, $0\in \overline{\{x_{k}\}_k}\subseteq \cup_{n\in S} \overline{U_n}$ that contradicts (a).
\end{proof}


Let $(G,+)$ be an abelian group, and let $X$ be a set. The group $G$ acts on $X^G$ by a shift as follows:
\[
s: G\times X^G \to X^G,\ s(g,f)(h)=f(g+h),
\]
where $f\in X^G$ and $g,h\in G$. An element $f\in X^G$ is called a \emph{Korovin mapping} if for any $M\subseteq G$ with $|M|\leq\w$ and each $h\in X^M$, there exists an element $g=g_h\in G$ such that $h=s(g,f){\restriction}_M$. For a Korovin mapping $f$, the subspace
\[
G_f = \{s(g,f): g\in G\} \; \mbox{ of }\; X^G,
\]
is called a \emph{Korovin orbit}. For $M\subseteq G$, let $\pi_M: X^G\to X^M$ be the projection. If $M=\{\alpha\}$, we set $\pi_\alpha=\pi_{\{\alpha\}}: X^G\to X$, so that $\pi_\alpha(f)=f(\alpha)$ for every $f\in M^G$.

\begin{lemma}\label{l:korovin}
Let $f$ be a Korovin map. Then for any countable $A\subseteq G_f$ and each mapping $\varphi: A\to X$, there is an $\alpha\in G$ such that $\varphi=\pi_\alpha{\restriction}_A$.
\end{lemma}

\begin{proof}
Let $A=\{f_n:n\in\NN\}$, where all $f_n$ are distinct. Then $f_n=s(\alpha_n,f)$ for some $\alpha_n\in G$; as $f_n\not= f_m$ it follows that also $\alpha_n\not= \alpha_m$ for all distinct $n,m\in\NN$. Set $M:=\{\alpha_n\}_{n\in\NN}$, and define $h\in X^M$ by $h(\alpha_n):=\varphi(f_n)$ for $n\in\NN$. Since $f$ is a Korovin map, for $M$ and $h$, there is $\alpha\in G$ such that $h=s(\alpha,f){\restriction}_M$. Hence, for every $n\in\NN$, we have
\[
\varphi(f_n)=h(\alpha_n)=s(\alpha,f)(\alpha_n)=f(\alpha+\alpha_n)=f(\alpha_n+\alpha)=s(\alpha_n,f)(\alpha)=f_n(\alpha)=\pi_\alpha(f_n).
\]
Thus $\varphi=\pi_\alpha{\restriction}_A$.
\end{proof}

Let $X$ be a Tychonoff space. A subset $B$ of $X$ is {\em $C$-embedded} if any continuous function $f:B\to\IR$ can be extended to a continuous function $\bar f:X\to\IR$. Recall also that a subset $Y$ of $X$ is {\em $G_\delta$-dense} in $X$ if $Y$ intersects any non-empty $G_\delta$-set of $X$.

\begin{proposition}\label{p:P-G_delta-dense}
Let $G$ be a set of cardinality $2^\omega$. Then there exists $P\subseteq \IR^G$ such that
\begin{enumerate}
\item[{\rm(i)}] $P$ is $G_\delta$-dense in $\IR^G$;
\item[{\rm(ii)}] every countable subset of $P$ is closed, discrete, and $C$-embedded in $P$;
\item[{\rm(iii)}]  for any countable $Q\subseteq P$ and each function $g: Q\to \IR$, there exists an $\alpha\in G$ such that $g=\pi_\alpha{\restriction}_Q$.
\end{enumerate}
\end{proposition}
\begin{proof}
We can assume that $G$ is an abelain group of cardinality $2^\w$. It follows from \cite[Theorem 1]{ReznichenkoTkachenko2024} that there is a Korovin map $f:G\to \IR$. Let $P=G_f$ be the Korovin orbit.
We show that $P$ is as desired.

(i) By \cite[Proposition 1]{ReznichenkoTkachenko2024}, the subset $P$ of $\IR^G$ {\em fills countable subproducts} in $\IR^G$ in the sense that  $\pi_M(P)=\IR^M$, for each countable set $M \subseteq G$. Then, by \cite[Proposition~2.16]{BG-Baire}, the space $P$ is $G_\delta$-dense in $\IR^G$.

(ii) It follows from \cite[Theorem 2]{ReznichenkoTkachenko2024} that every countable subset of $P$ is closed, discrete, and $C$-embedded in $P$.

(iii) follows from Lemma \ref{l:korovin} applied to $X=\IR$.
\end{proof}

\begin{example}\label{lcb-not_kFU}
There exists a locally convex space $E$ such that:
\begin{enumerate}
\item[{\rm(i)}] $E$ is a Baire space;
\item[{\rm(ii)}]  $E$ is feral, and hence $E$ is quasi-complete and has the property $(K)$;
\item[{\rm(iii)}]  $E$ is not $\kappa$-Fr\'{e}chet--Urysohn;
\item[{\rm(iv)}] the completion $\widehat E$ of $E$ is a $\kappa$-Fr\'{e}chet--Urysohn space.
\end{enumerate}
\end{example}

\begin{proof}
Let $G$ be a set of cardinality $2^\omega$. It follows from Proposition \ref{p:P-G_delta-dense} that there is a $G_\delta$-dense subset $P$ in $\IR^G$ such that the following assertion holds.

\smallskip

{\em {\bf Claim 1.} For any countable $Q\subseteq P$ and each function $g: Q\to \IR$, there exists an $\alpha\in G$ such that $g=\pi_\alpha{\restriction}_Q$.}

\smallskip

(i)
The space $P$ is $G_\delta$-dense  in the Baire space $\IR^G$ (Corolary \ref{c:product-Frechet-Baire}) and, therefore, by \cite[Corollary~2.20]{BG-Baire}, $P$ is a Baire space.
Claim 1 implies that $P$ is a linearly independent subset of $\IR^G$. Let $E$ be the linear span of $P$. Then $P$ is a Hamel basis of $E$. Since $P$ is dense in $\IR^G$ and hence also in $E$, the space $E$ is Baire.

\smallskip

(ii) Suppose for a contradiction that there is a countably infinite bounded set $\{x_n\}_{n\in\NN}\subseteq E$ which is not contained in a finite-dimensional subspace. Fix a countable $Q\subseteq P$ such that the sequence $\{x_n\}_{n\in\w}$ lies in the linear span of $Q$. For every $n\in\NN$, there exists a finite $Q_n\subseteq Q$ and a function $\mu_n: Q_n\to\IR$ such that $0\notin \mu_n(Q_n)$ and $x_n=\sum_{q\in Q_n}\mu_n(q) q$. We extend $\mu_n$ onto $Q$  setting $\mu_n(q)=0$ for every $q\in Q\SM Q_n$. Since $\{x_n\}_{n}$ is not contained in a finite-dimensional subspace, $|\bigcup_n Q_n|=\w$. Then there exists $(n_k)_{k\in\NN}\subseteq\NN$ such that
\[
R_k :=Q_{n_k}\setminus\bigcup_{i=1}^{k-1} Q_{n_i}\neq \emptyset \; \mbox{ for every $k>1$}.
\]
Fix $q_1\in Q_{n_1}$ and $q_k\in R_k$ for $k>1$.

Taking into account that $\mu_{n_k}(q_k)\neq 0$, by induction on $k$, we can define a sequence $(\lambda_k)_k\subseteq\IR$ such that
\[
v_k := \sum_{i=1}^k \lambda_i \mu_{n_k}(q_i) > k \quad \mbox{ for every $k\in\NN$}.
\]
Define a function $g: Q\to\IR$ by
\[
g(q)=\begin{cases}
\lambda_i,& \text{if }q=q_i\text{ for some }i\in\NN,\\
0, & \text{otherwise}.
\end{cases}
\]
Claim 1 implies that there exists $\alpha\in G$ such that $g=\pi_\alpha{\restriction}_Q$. Then
\[
\pi_\alpha(x_{n_k})=\pi_\alpha\Big( \sum_{q\in Q_{n_k}}\mu_{n_k}(q) q\Big)= \sum_{q\in Q_{n_k}}\mu_{n_k}(q)\cdot g(q) =v_k>k
\]
for every $k\in\NN$. Therefore, the continuous functional $\pi_\alpha$ is unbounded on the bounded set $\{x_n\}_{n\in\w}$, a contradiction.

The space $E$ is quasi-complete and has the property $(K)$ by Proposition \ref{p:feral-K}.
\smallskip

(iii) Since $E$ is dense in $\IR^G$ and $E$ is feral, it follows from Proposition \ref{p:feral-not-kFU} that $E$ is not a $\kappa$-Fr\'{e}chet--Urysohn space.
\smallskip

(iv) The completion $\widehat E$ of $E$ is $\IR^G$.
Since an arbitrary power of $\IR$ is $\kappa$-Fr\'{e}chet--Urysohn by  \cite[Theorem 3.1]{Mrowka}, $\widehat E$ is a $\kappa$-Fr\'{e}chet--Urysohn space.
\end{proof}

Example \ref{lcb-not_kFU} shows that the quasi-completeness of a Baire lcs $E$ does not imply that $E$ is $\kappa$-Fr\'{e}chet--Urysohn.
\begin{problem}
Let $E$ be a complete Baire lcs. Is $E$ a $\kappa$-Fr\'{e}chet--Urysohn space?
\end{problem}


\section{Baire-like and $b$-Baire-like properties of $\kappa$-Fr\'{e}chet--Urysohn locally convex spaces  } \label{sec:kFU-Baire}


In this section we characterize $\kappa$-Fr\'{e}chet--Urysohn locally convex spaces which are Baire-like, and show that such spaces are $b$-Baire-like. First we recall basic definitions. A subset $A$ of topological vector space $E$ is {\em bornivorous} if for every bounded subset $B$ of $E$, there exists $n\in\w$ such that $B\subseteq nA$. A locally convex space $E$ is called
\begin{enumerate}
\item[$\bullet$] (Ruess \cite{Ruess-76}) {\em $b$-Baire-like} if for any increasing sequence $\{A_n\}_{n\in\w}$ of closed absolutely convex and bornivorous subsets of $E$ covering $E$, there is an $n\in\w$ such that $A_n$ is a neighborhood of zero;
\item[$\bullet$] (Saxon \cite{Saxon-Baire}) {\em Baire-like} if for any increasing sequence $\{A_n\}_{n\in\w}$ of closed absolutely convex subsets of $E$ covering $E$, there is an $n\in\w$ such that $A_n$ is a neighborhood of zero.
\end{enumerate}
The relationships between the introduced notions are described in the next diagram in which none of the implications is invertible (for more details, see \cite[Section~2.4]{kak} and \cite[Section~1.2]{FLPSR})
\[
\xymatrix{
\mbox{Baire} \ar@{=>}[r] & \mbox{Baire-like}  \ar@{=>}[rd]  \ar@{=>}[r] & \mbox{$b$-Baire-like} \ar@{=>}[r]  & \mbox{quasibarrelled}\\
& & \mbox{barrelled} \ar@{=>}[ru] &}.
\]
It is known that for each Tychonoff space $X$, the space $C_p(X)$ is $b$-Baire-like \cite[Corollary~2.8]{kak}, and the space $B_1(X)$ is Baire-like \cite[Theorem~1.1]{BG-Baire-lcs}. Recall that the space $B_1(X)$ of all Baire-one functions on $X$ is the family of all real-valued  functions on $X$ which are limits of sequences in $C_p(X)$.





Our first and somewhat unexpected result shows that $\kappa$-Fr\'{e}chet--Urysohness of an lcs $E$ implies that $E$ is $b$-Baire-like.
\begin{theorem} \label{t:kFU=>b-Baire-like}
Each $\kappa$-Fr\'{e}chet--Urysohn lcs $E$ is $b$-Baire-like. Consequently, $E$ is quasibarrelled.
\end{theorem}

\begin{proof}
Let $\{A_n\}_{n\in\w}$ be an increasing sequence of closed absolutely convex and bornivorous subsets of $E$ covering $E$. Suppose for a contradiction that none of the sets $A_n$ is a neighborhood of zero in $E$. For each $n\in\w$, set $U_n:=E\SM (nA_n)$. Then $U_n$ is an open subset of $E$ such that $0\in \overline{U_n}$ (because $\Int(nA_n)=\emptyset$). Since  $E$ is $\kappa$-Fr\'{e}chet--Urysohn, (i) of Theorem \ref{t:kFU-group-charac} implies that there are a strictly increasing sequence $(n_k)\subseteq \w$ and a sequence $\{x_k\}_{k\in\w}$ in $E$ such that $x_k\in U_{n_k}$ $(k\in\w)$ and $x_k\to 0$.
By construction, the set $K:=\{x_k\}_{k\in\w}\cup\{0\}$ is compact and, for every $n\in\w$, the intersection $nA_n\cap K$ is finite.

On the other hand, since $K$ is compact and hence bounded, there is an $n\in\w$ such that $K\subseteq nA_0$. Since $\{A_n\}_{n\in\w}$ is increasing, we obtain that $K\subseteq nA_n$. Since $nA_n\cap K$ must be finite, we obtain a contradiction. Thus $E$  is a $b$-Baire-like space.
\end{proof}


Every metrizable lcs is $b$-Baire-like, see \cite[Proposition~2.11]{kak}. However, Example \ref{exa:metr-not-K-Baire} shows that there are metrizable (hence $\kappa$-Fr\'{e}chet--Urysohn) locally convex spaces which are not Baire-like. (Indeed, it is easy to see that the space $E$ from Example \ref{exa:metr-not-K-Baire} is not barrelled, and hence it is not Baire-like.) In fact this example shows that to obtain the stronger property of being Baire-like we need some additional property in the completion of the space, as for example, of being ``massive'' in the sense of (iii) of the next theorem.

\begin{theorem} \label{t:kFU-Baire-like}
For a $\kappa$-Fr\'{e}chet--Urysohn lcs $E$, the following assertions are equivalent:
\begin{enumerate}
\item[{\rm(i)}] $E$ is Baire-like;
\item[{\rm(ii)}]  $E$ is barrelled;
\item[{\rm(iii)}] for every increasing sequence $\{A_n\}_{n\in\w}$ of closed absolutely convex subsets of $E$ covering $E$, it follows that  $\overline{E}=\bigcup_{n\in\w} \overline{A_n}$, where $\overline{E}$ is the completion of $E$.
\end{enumerate}
\end{theorem}

\begin{proof}
(i)$\Ra$(ii) is clear. 

(ii)$\Ra$(iii) immediately follows from Valdivia's Proposition 2.13 of \cite{kak}.

(iii)$\Ra$(i) Let $\{A_n\}_{n\in\w}$ be an increasing sequence of closed absolutely convex subsets of $E$ covering $E$. For every $n\in\w$, let $\overline{A_n}$ be the closure of $A_n$ in $\overline{E}$. Note that $\{\overline{A_n}\}_{n\in\w}$ is an increasing sequence of closed absolutely convex subsets of $\overline{E}$. By (iii), we have $\overline{E}=\bigcup_{n\in\w} \overline{A_n}$. Since $E$ is dense in $\overline{E}$, Corollary 2.2 of \cite{Gabr-B1} implies that also $\overline{E}$ is $\kappa$-Fr\'{e}chet--Urysohn. As $\overline{E}$ is complete and hence locally complete, Corollary \ref{c:lc-kFU-Baire} implies that $\overline{E}$ is a Baire space. Therefore there is an $m\in\w$ such that $\overline{A_m}$ is a neighborhood of the origin in $\overline{E}$. Hence $A_m=\overline{A_m}\cap E$ is a neighborhood of zero in $E$. Thus $E$ is Baire-like.
\end{proof}


\section{Applications to spaces of continuous functions and Baire-one functions} \label{sec:B1}


In this section, applying our main results, we provide simple proofs of some highly non-trivial results.
Recall that  Theorem~3.1 of \cite{Osipov-252} states that $C_p(X)$ is $\kappa$-Fr\'{e}chet-Urysohn if, and only if, $B_1(X)$ is Baire. The necessity in this result is an easy corollary of Theorem \ref{t:kFU_rsc-Baire}.

\begin{proposition}[\protect{\cite{Osipov-252}}]\label{p:kappa-baire1}
If $C_p(X)$ is $\kappa$-Fr\'{e}chet--Urysohn, then $B_1(X)$ is Baire.
\end{proposition}
\begin{proof}
Since $C_p(X)$ is a relatively sequentially complete and dense subset of $B_1(X)$, Theorem \ref{t:kFU_rsc-Baire} implies that $B_1(X)$ is Baire.
\end{proof}

For any countable ordinal $\alpha>1$, let $B_\alpha(X)$ be the family of all functions $f:X\to \IR$ which are pointwise limits of sequences from $\bigcup_{\beta<\alpha}B_\beta(X)$. The space $B_\alpha(X)$ is the space of all Baire-$\alpha$ functions and is endowed with the topology induced from the product $\IR^X$.

The second application of Theorem \ref{t:kFU_rsc-Baire} is the next result proved in Corollary~4.2 of \cite{BG-Baire}.

\begin{proposition}[\protect{\cite{BG-Baire}}]\label{p:baire2}
For each Tychonoff space $X$ and for every countable ordinal $\alpha>1$, the space $B_{\alpha}(X)$ is a  Baire space.
\end{proposition}
\begin{proof}
Since $B_1(X)$ is a relatively sequentially complete and dense subset of $B_2(X)$ and $B_1(X)$ is $\kappa$-Fr\'{e}chet--Urysohn  \cite[Theorem 5.15]{Gabr-seq-Ascoli}, Theorem \ref{t:kFU_rsc-Baire} implies that $B_2(X)$ is a Baire space. Since $B_2(X)$  is a dense subspace of $B_\alpha(X)$, it follows that $B_{\alpha}(X)$ is a  Baire space.
\end{proof}

Below we give applications to the space $\CC(X)$.
Let $X$ be a Tychonoff space. According to \cite{mcoy}, a family $\AAA$ of non-empty subsets of a space $X$ is said to be {\em moving off} if for each compact set $K\subseteq X$, there exists $A\in\AAA$ such that $K\cap A=\emptyset$.

A  family $\AAA$ of non-empty subsets of a space $X$ is {\em compact-finite} if for every compact set $K\subseteq X$, the family $\{A\in\AAA: A\cap K\not=\emptyset\}$ is finite. A  family $\{A_i\}_{i\in I}$ of non-empty subsets of a space $X$ is said to be {\em strongly compact-finite} if for every $i\in I$, there is an open set $U_i$ such that $A_i\subseteq U_i$ and the family $\{U_i\}_{i\in I}$ is compact-finite. The following characterization of $\kappa$-Fr\'{e}echet--Urysohn spaces $\CC(X)$ was obtained by Sakai \cite[Theorem~2.3]{Sakai}.

\begin{theorem}[\cite{Sakai}] \label{t:Ck-kFU}
For a Tychonoff space $X$, the space $\CC(X)$ is $\kappa$-Fr\'{e}chet--Urysohn if, and only if, $X$ satisfies the following property:
\begin{enumerate}
\item[$(\kappa_k)$] every moving off family of compact subsets of $X$ has a countable subfamily which is strongly compact-finite.
\end{enumerate}
\end{theorem}

Theorem \ref{t:Ck-kFU} and Theorem \ref{t:kFU=>b-Baire-like} immediately imply the next corollary.
\begin{corollary} \label{c:Ck-b-Baire-like}
If a Tychonoff space $X$ has the property $(\kappa_k)$, then $\CC(X)$ is $b$-Baire-like.
\end{corollary}
We do not know a characterization of spaces $X$ for which $\CC(X)$ is a $b$-Baire-like space. On the other hand, Lehner \cite[Proposition~2.18]{kak} characterized Baire-like spaces $\CC(X)$.

To characterize Tychonoff spaces $X$ for which $\CC(X)$ is a Baire space is an old open problem. Some necessary and sufficient conditions for enough narrow classes of spaces were obtained by McCoy and Ntantu \cite[Chapter~5.3]{mcoy} and Gruenhage and Ma \cite{GruenhageMa1997}.
Below, under some additional condition on $\CC(X)$, we characterize Baire spaces $\CC(X)$.
\begin{theorem}\label{t:Ck_MK}
If $\CC(X)$ has the property $(MK)$, then $\CC(X)$ is a Baire space if, and only if, $X$ satisfies condition $(\kappa_k)$.
\end{theorem}

\begin{proof}
If $\CC(X)$ is a Baire space, then, by Proposition 2.5(1) of \cite{Sakai} and Theorem \ref{t:Ck-kFU},  $X$ satisfies condition $(\kappa_k)$. Conversely, if  $X$ satisfies condition $(\kappa_k)$, then, by Theorem \ref{t:Ck-kFU}, $\CC(X)$ is $\kappa$-Fr\'{e}chet--Urysohn, and hence, by Theorem \ref{t:MK-Baire},  $\CC(X)$ is a Baire space.
\end{proof}

Theorem \ref{t:Ck_MK} motivates the following problem.
\begin{problem}
Characterize Tychonoff spaces $X$ for which $\CC(X)$ has the property $(MK)$.
\end{problem}

A characterization of local completeness of $\CC(X)$ in terms of $X$ is given in Theorem 2.1 of \cite{Gabr-lc-Ck}. It is a classical result that $\CC(X)$ is complete if, and only if, $X$ is a $k_\IR$-space, see \cite[Corollary~5.1.2]{mcoy}. A space $X$ is a {\em $k_\IR$-space} if every $k$-continuous function $f:X\to\IR$ is continuous ($f$ is {\em $k$-continuous} if each restriction of $f$ to any compact set $K\subseteq X$ is continuous). Now Theorem \ref{t:Ck_MK} and Proposition \ref{p:lc-k-complete} immediately imply the following assertion.

\begin{corollary} \label{c:Ck-Baire-kR}
If a Tychonoff space $X$ is such that $\CC(X)$ is locally complete {\rm(}for example, $X$ is a $k_\IR$-space{\rm)}, then $\CC(X)$ is a Baire space if, and only if, $X$ satisfies condition  $(\kappa_k)$.
\end{corollary}

%






\bibliographystyle{amsplain}

\end{document}